\documentclass{amsart}
\usepackage{amsmath, amssymb,epic,graphicx,mathrsfs,enumerate}
\usepackage[all]{xy}
\usepackage{color}
\usepackage{comment}

\usepackage{amsthm}
\usepackage{amssymb}
\usepackage{latexsym}
\usepackage{longtable}
\usepackage{epsfig}
\usepackage{amsmath}
\usepackage{hhline}

\newtheorem{teor}{Theorem}

\newtheorem{lemma}{Lemma}

\DeclareMathOperator{\alt}{Alt}

\DeclareMathOperator{\Aut}{Aut}

\DeclareMathOperator{\soc}{soc}

\DeclareMathOperator{\frat}{Frat}
\DeclareMathOperator{\GL}{GL}
\DeclareMathOperator{\SL}{SL}
\DeclareMathOperator{\PSL}{PSL}

\newcommand{\qbin}[2]{\begin{bmatrix}{#1}\\ {#2}\end{bmatrix}_p}

\begin{document}
		\bibliographystyle{amsplain}
\title[The genus of the subgroup graph of a finite group]{The genus of the subgroup graph\\  of a finite group}
\date{}

\author{Andrea Lucchini}
\address{Dipartimento di Matematica \lq\lq Tullio Levi-Civita", Universit\`a di Padova, Via Trieste 63, 35131 Padova, Italy}
\email{lucchini@math.unipd.it}

\begin{abstract}
For a finite group $G$ denote  by $\gamma(L(G))$ the genus of the subgroup graph of $G.$  We prove that $\gamma(L(G))$ tends to infinity as either the rank of $G$ or the number of prime divisors of $|G|$ tends to infinity.
\end{abstract}

\maketitle

\section{Introduction}
The subgroup graph $L(G)$ of a finite group $G$ is the graph whose vertices are the subgroups
of the group and two vertices, $H_1$ and $H_2,$ are connected by an edge if and only if $H_1 \leq H_2$ and there is no subgroup $K$ such that $H_1\leq K \leq H_2.$  A graph is said to be embedded in a surface $S$ when it is drawn on $S$ so that no two edges intersect.

\

 The genus $\gamma(\Gamma)$ of a graph $\Gamma$ is the minimum $g$ such that there exists an embedding of $\Gamma$ into the orientable surface $S_g$ of genus $g$ (or in other words the minimun number $g$ of handles which must be added to a sphere so that $G$ can be embedded on the resulting surface).
 
\

In \cite{pla} the authors investigate the case $\gamma(L(G))=0$, characterizing the finite groups having a planar subgroup graph. It turns out that there are seven infinite families of such groups, and three additional isolated groups. All these groups have order divisible by at most three different primes and their Sylow subgroups have rank at most 2 (recall that the rank of a finite group $G$	 is the minimal number $r$ such that every subgroup of $G$ can be generated by $r$ elements). In this paper we generalize this result proving that for every non-negative integer $k,$ there exist two integers $a_k$ and $b_k$ such that if $\gamma(L(G))\leq k,$ then the order of $G$ is divisible by at most $a_k$ different primes and the rank of any Sylow $p$-subgroup of $G$ is at most $b_k.$ Let  $r_p(G)$ be the rank of a Sylow $p$-subgroup of $G,$
$r(G)$ the rank of $G$, $\pi(G)$ the set of the prime divisors of $|G|$ and $\rho(G)=\max_{p\in \pi(G)}r_p(G).$ By the main result in \cite{rg} and \cite{al}, if $r_p(G)\leq d$ for every $p\in \pi(G),$ then $r(G)\leq d+1$ and therefore $\rho(G)\leq r(G)\leq \rho(G)+1.$ Hence we may state our results as follows:

\begin{teor}\label{main}Let $G$ be a finite group. Then $\gamma(L(G))$ tends to infinity as either  the rank of $G$ or the number of prime divisors of $|G|$ tends to infinity.
\end{teor}

Our proof uses the classification of the finite non-abelian simple groups. In particular we prove the following result, of independent interest.

\begin{teor}\label{simple} For every $k\in \mathbb N,$ there exist only finitely many non-abelian finite simple groups $S$, with $\gamma(L(S))\leq k.$
\end{teor}

Notice that the previous theorem cannot be deduced from Theorem \ref{main}. For example $\rho(\PSL(2,p))=2$ for every prime $p$ and it follows from \cite[Corollary 4.2]{tur} that there are infinitely many primes $p$ such that $|\PSL(2,p)|$ is divisible by at most 20 primes.

\section{Proofs of our results}
If $\Omega$ is a family of subgroups of $G,$ we may consider the graph $L_\Omega(G)$ whose vertices are the subgroups in $\Omega$ and two vertices, $H_1$ and $H_2,$ are connected by an edge if and only if $H_1 \leq H_2$ and there is no subgroup $K\in \Omega$ such that $H_1\leq K \leq H_2.$ Clearly  $\gamma(L_\Omega(G))\leq \gamma(L(G))$ for every choice of $\Omega.$
Moreover, for any choice of $\Omega$, the graph $\gamma(L_\Omega(G))$ is triangle-free.
An easy consequence of Euler's formula (see for example \cite[Corollary 11.17(b)]{hara}) is that if $\Gamma$ is triangle-free then 
$$\gamma(\Gamma)\geq \frac{|E(\Gamma)|}{4}-\frac{|V(\Gamma)|}{2}+1.$$
So we have:
\begin{lemma}\label{ineq}
Let $G$ be a finite group and let $\Omega$ be a family of subgroups of $G.$ Then
$$\gamma(L(G))\geq \gamma(L_\Omega(G))\geq\frac{|E(L_\Omega(G))|}{4}-\frac{|V(L_\Omega(G))|}{2}+1.$$
\end{lemma}

\begin{lemma}\label{ciclo}
	Let $G$ be a finite soluble group and let $t$ be the number of the distinct prime divisors of $|G|$. Then $$\lim_{t\to \infty}\gamma(L(G))=\infty.$$
\end{lemma}

\begin{proof}
	Suppose that $p_1,\dots,p_t$ are the distinct prime divisors of the order of $G$ and let $P_1,\dots,P_t$ be a Sylow basis of $G.$ To any subset $J$ of $\{p_1,\dots,p_t\},$ there corresponds the  subgroup $H_J=\prod_{j\in J}P_j$ of $G.$ Let $\Omega=\{H_J\mid J \subseteq \{1,\dots,t\}\}$. The vertices of $L_\Omega(G)$ correspond to the subsets of $\{p_1,\dots,p_t\},$ so $|V(L_\Omega(G)|=2^t.$ If $J_1,J_2\subseteq \{p_1,\dots,p_t\}$ and
	$|J_1|\leq |J_2|,$ then $H_{J_1}$ and $H_{J_2}$ are adjacent if and only if $J_2=J_1\cup \{p\}$ for a prime $p\notin J_1,$ so 
	$$|E(L_\Omega(G))|=\sum_{0\leq i\leq t-1}\binom t i (t-i).$$
	 It follows from Lemma \ref{ineq} that
	$$\begin{aligned}\gamma(L(G))&\geq\sum_{0\leq i\leq t-1}\binom t i \frac{(t-i)}{4}-\frac{2^t}{2}+1\geq \sum_{0\leq i\leq \frac{\lfloor t-1\rfloor}2 }\binom t i \frac{t}{8}-\frac{2^t}{2}+1\\&\geq \frac{2^t\cdot t}{16}-\frac{2^t}{2}+1=2^{t-1}\left(\frac{t}{8}-1\right)+1.\qedhere
	\end{aligned}
	$$
\end{proof}

Notice that the assumption $t\geq 3$ in the statement of the previous lemma is necessary. Indeed $L(A_{p,2})\cong K_{2,p+1}$ is a planar graph for any choice of $p.$

\begin{lemma}\label{rango} Let $p$ be a prime, $t$ a positive integer and $A_{p,t}$ the elementary abelian $p$-group of rank $t.$ If $t\geq 3,$ then $\gamma(L(A_{p,t}))$ tends to infinity as $p^t$ tends to infinity.
\end{lemma}

\begin{proof} Let $\Omega$ be the family of the subgroups of $A_{p,t}$ of order $p$ and $p^2$
	and let $\Gamma=\Gamma_\Omega(A_{p,t}).$
	Notice that $\Gamma$ is a bipartite graph and
	$$|V(\Gamma)|=\qbin{t}{2}+\qbin{t}{1},\quad  |E(\Gamma)|=\qbin{t}{2}\frac{p^2-1}{p-1}.$$
	Since $$\qbin{t}{1}=\qbin{t}{2}\frac{p^2-1}{p^{t-1}-1},$$
	we deduce
	$$\begin{aligned}\gamma(L(G))&\geq \gamma(\Gamma)\geq \frac{|E(\Gamma)|}{4}-\frac{|V(\Gamma)|}{2}\\&=
	\frac{1}{4}
	\qbin{t}{2}
	\left(
	\frac{p^2-1}{p-1}-
	2\left(1+\frac{p^2-1}{p^{t-1}-1}\right)
	\right)\\
	&= 	\frac{1}{4}
	\qbin{t}{2}
	\left(p-1-\frac{2(p^2-1)}{p^{t-1}-1}\right).
	\end{aligned}
	$$
	In particular, if $t\geq 3,$ then
 $\gamma(L(A_{p,t}))$ tends to infinity as $p^t$ tends to infinity.
\end{proof}

\begin{lemma}\label{pls}  $\gamma(L(\PSL(2,q)))$ tends to infinity as $q$ tends to infinity.
\end{lemma}
\begin{proof}
First assume $q=2^t.$ In this case a Sylow 2-subgroup of $\PSL(2,q)$ is isomorphic to $A_{2,t}$, hence
$\gamma(L(\PSL(2,2^t))\geq \gamma(L(A_{2,t}))$ and the conclusion follows from Lemma \ref{rango}.
Now assume that $q$ is odd. If $q\notin \{5,7,9,17\}$, then there exists $n\in \{\frac{q-1}2,\frac{q+1}2\}$ such that $n$ is divisible by at least two different primes (see for example \cite[Theorem 3]{mh}). We factorize $n=a\cdot b$ where $a$ and $b$ are coprime integers properly dividing $n.$ The group $\PSL(2,q)$ has a maximal subgroup $M$ isomorphic to the dihedral group $D_n$ of order $2n,$ and inside $M$ we can find $n/a$ subgroups $H_1,\dots,H_{n/a}$ isomorphic to $D_a$, $n/b$ subgroups $K_1,\dots,K_{n/b}$ isomorphic to $D_b$ and $n$ subgroups $J_1,\dots,J_n$ that have order 2 and are non-central in $M.$ Let $\Omega:=\{M, H_i, K_j, J_k, 1\mid 1\leq i\leq \frac n a, 1\leq j\leq \frac n b, 1\leq k\leq n\}.$ We have
$$\begin{aligned}v&=|V(\Gamma_\Omega(\PSL(2,q))|=2+n+\frac n a+\frac n b,\\
e&=|E(\Gamma_\Omega(\PSL(2,q))|=3n+\frac n a+\frac n b,
\end{aligned}$$
since $M$ is adjacent to  $H_i$ and $K_j$ for $1\leq i\leq n/a$ and $1\leq j\leq n/b,$ 1 is adjacent to  $J_k$ for $1\leq k\leq n,$
any $H_i$ contains precisely $a$ non-central subgroups of order 2 and any $K_j$ contains precisely $b$ non-central subgroups of order 2. 
It follows
$$\begin{aligned}\gamma(L(\PSL(2,q)))&\geq \gamma(L_\Omega(\PSL(2,q)))\geq \frac e 4 - \frac v 2 +1\\&=\frac{1}{4}\left(3n+\frac{n}{a}+\frac{n}{b}\right)-\frac{1}{2}\left(2+n+\frac{n}{a}+\frac{n}{b}\right)+1\\&=\frac{n}{4}\left(1-\frac{1}{a}-\frac{1}{b}\right)\geq \frac{n}{4}\left(1-\frac{1}{2}-\frac{1}{3}\right)\geq \frac{n}{24}.
\end{aligned}
$$
The conclusion follows from the observation that $n$ tends to infinity as $q$ tends to infinity.
\end{proof}

\begin{lemma}\label{beta}
If $G\leq \GL(2,p)$, where $p$ is a prime, and $\gamma(L(G))\leq k,$ then the number of the prime divisors of $|G|$ is at most $\beta_k,$ where $\beta_k$ is a positive integer depending only on $k.$
\end{lemma}
\begin{proof}By Lemma \ref{ciclo}, there exists $\alpha_k$ such that if $X$ is a finite soluble group and $\gamma(L(X))\leq k,$ then $|\pi(X)|\leq \alpha_k.$ So may so assume that $G$ is not soluble. If follows from \cite[Hauptsatz 8.27]{hu} 
that there are two cases:

\noindent a) $\PSL(2,p)$ is a composition factor of $G.$ In this case $\gamma(L(X))\leq k$ for every subgroup $X$ of $\PSL(2,p).$ Since $\PSL(2,p)$ contains two soluble subgroups $H_1$ and $H_2$ of order, respectively, $p+1$ and $p(p-1)/2$ we deduce from Lemma \ref{ciclo} that $|\pi(G)|=|\pi(H_1)\cup \pi(H_2)|\leq 2\alpha_k.$

\noindent b) $\alt(5)$ is a composition factor of $G$. In this case $$\frac{G\cap \SL(2,p)}{G\cap Z(\SL(2,p))} \cong \alt(5),$$ 
so $G$ contains a normal subgroup $N$ such that $G/N$ is cyclic and $\pi(N)=\{2,3,5\}.$ It follows 
$|\pi(G)|\leq \alpha_k+3.$
\end{proof}

\begin{proof}[Proof of Theorem \ref{simple}]
	Fix $k\in \mathbb N$ and let $S$ be a finite non-abelian simple group with $\gamma(L(S))\leq k.$	Since $r_2(\alt(n))\geq \lfloor \frac n 2 \rfloor -1,$ it follows from Lemma \ref{rango}, that $\gamma(L(\alt(n))>k$ if $n$ is large enough. So we may assume that $S$ is of Lie type over the field $\mathbb F_q,$ $q=p^f.$ Since $r_p(S)$ tends to infinity as the Lie rank of $S$ tends to infinity, it follows from Lemma \ref{rango} that the Lie rank of $S$ is bounded in term of $k.$ On the other hand $S$ contains a section isomorphic to $\PSL(2,q)$, so,
	by Lemma \ref{pls}, $\gamma(L(S))\geq \gamma(L(\PSL(2,q))$ tends to infinity as $q$ tends to infinity. So we bounded in term of $k$ either  $q$ as the Lie rank, and consequently the number of possibilities for $S$ itself.
\end{proof}

\begin{proof}[Proof of Theorem \ref{main}]
By Lemma \ref{rango}, 
there exists $\rho_k$ such if $\gamma(L(G))\leq k,$ then $r_p(G)\leq \rho_k$ for every prime divisor $p$ of $|G|,$ and consequently $r(G)\leq \rho_k+1.$ So it suffices to prove that there  exists $\pi_k$ such if $\gamma(L(G))\leq k,$ then $|\pi(G)|\leq \pi_k.$ 
		Since $\pi(G)=\pi(G/\frat(G))$ (see \cite[Satz 3.8]{hu}) and $\gamma(L(G))\geq \gamma(L(G/\frat(G))$ we may assume $\frat(G)=1.$ 
		Write $$X=\soc(G)=A_1\times \dots \times A_r\times B_1\times \dots \times B_s\times C_1\times \dots \times C_t$$ as a direct product of minimal normal subgroups of $G$ with the property that the factors $A_i$'s are abelian with rank at most 2, the factors $B_j$'s are abelian with rank at least 3 and the factors $C_l$'s are non-abelian. By Theorem \ref{simple}, the family $\mathcal S_k$ of the finite non-abelian simple groups $S$ with $\gamma(L(S))\leq k$ is finite. Let $N$ be a minimal non-abelian normal subgroup of $G.$ There exists a non-abelian simple group $S$ and a positive integer $m$ such that $N\cong S^m.$ Since $\gamma(L(S))\leq \gamma(L(N))\leq \gamma(L(G))\leq k$ and $r_2(N)\geq 2\cdot m,$ if follows that $S\in \mathcal S_k$ and, by Lemma \ref{rango}, $m \leq \tau$ for a positive integer $\tau$ depending only on $k.$ Moreover by Lemma \ref{rango}, if $N\cong C_{p^u}$ is an abelian minimal normal subgroup of $G,$ then either $u\leq 2$ or $p^u\leq \sigma$ for a positive integer $\sigma$ depending only on $k.$ It follows that there exists a finite family $\mathcal F_k$ of finite characteristically simple groups such that $B_j, C_l\in \mathcal F_k$ for every $1\leq j\leq s$ and $1\leq l\leq t.$ Let $\Lambda_k$ be the set of the primes dividing $|Y||\Aut Y|$ for some $Y\in \mathcal F_k$ and set $\lambda_k=|\Lambda_k|.$ Since $\frat(G)=1,$ we have that $X$ coincides with the generalized Fitting subgroup of $G$ and consequently $C_G(X) = Z(X)$ and every prime dividing $|G/Z(X)|$ divides $|G/C_G(N)|$ for some minimal normal subgroup $N$ of $G.$ It follows   that a prime $p$ dividing $|G|$ either divides $|A_i||G/C_G(A_i)|$ for some $1\leq i \leq r,$ or belong to $\Lambda_k.$ Let $\Sigma=\cup_{1\leq i\leq r}\pi(A_i)$. 
		It follows from Lemma \ref{ciclo}, that $|\Sigma|\leq \alpha_k$ for an integer $\alpha_k$ depending only on $k.$ If we denote by $\Sigma_i$ the set of the prime divisors of $|G/C_G(A_i)|,$ we have
$\pi(G)\subseteq \Sigma \cup \Lambda_k \cup_{1\leq i \leq r}\Sigma_i.
$
	If $A_i$ is cyclic, then $G/C_G(A_i)$ is cyclic 
	and again it follows from Lemma \ref{ciclo} that $|\Sigma_i|\leq \alpha_k.$ Otherwise $|\Sigma_i|\leq \beta_k$ by Lemma \ref{beta}. We deduce  that $|\pi(G)|\leq \alpha_k+\lambda_k+\alpha_k \max(\alpha_k, \beta_k).$
	\end{proof}

\end{document}